\RequirePackage{fix-cm}
\documentclass[smallextended]{svjour3}       
\smartqed  
\usepackage{amssymb,amsmath,multirow}
\usepackage{xcolor}

\newtheorem{thm}{Theorem}
\newtheorem{lem}{Lemma}
\newtheorem{cor}{Corollary}
\newtheorem{ass}{Assumption}

\begin{document}
\title{Strong Duality for Generalized Trust Region Subproblem:
S-Lemma with Interval Bounds
\thanks{This research was
supported by Beijing Higher Education Young Elite Teacher Project (29201442),
and by the fund of State Key Laboratory of
Software Development Environment (SKLSDE-2013ZX-13).}
}

\titlerunning{S-Lemma with Interval Bounds }        

\author{  Shu Wang  \and  Yong Xia }


\institute{
 S. Wang         \at
              School of
Mathematics and System Sciences, Beihang University, Beijing,
100191, P. R. China \email{wangshu.0130@163.com }
\and
Y. Xia \at
              State Key Laboratory of Software Development
              Environment, LMIB of the Ministry of Education,
              School of
Mathematics and System Sciences, Beihang University, Beijing,
100191, P. R. China
              \email{dearyxia@gmail.com}           
 }

\date{Received: date / Accepted: date}

\maketitle

\begin{abstract}
With the help of the newly developed S-lemma with interval bounds, we show that strong duality holds for the interval bounded generalized trust region subproblem under some mild assumptions, which answers an open problem raised by Pong and Wolkowicz [Comput. Optim. Appl. 58(2), 273-322, 2014].

\end{abstract}

\section{Introduction}

Consider the interval bounded
generalized trust region subproblem:
\begin{eqnarray*}
({\rm GTRS})~&\inf& f(x) \\
&{\rm s.t.}& \alpha\le h(x)\le \beta,
\end{eqnarray*}
where $\alpha\le \beta\in \Bbb R$, $f(x)$ and $h(x)$ are quadratic functions, i.e.,
\begin{eqnarray*}
f(x):=x^{T}Ax+2a^{T}x+c,\label{ff}\\
h(x):=x^{T}Bx+2b^{T}x+d,\label{fh}
\end{eqnarray*}
$A,B\in\Bbb R^{n\times n}$ are symmetric matrices, $a,b\in\Bbb R^{n}$, $c,d\in\Bbb R$.

When $B = I$, $b = 0$, and $\alpha\le0$, (GTRS) is known as the classical trust region subproblem (TRS), which
arises in  trust region methods for nonlinear programming \cite{CG}.
Though (TRS) is explicitly non-convex as $A$ is not necessarily positive semidefinite,
the necessary and sufficient optimality
condition  has been derived, see \cite{G81,M83}.
This makes sense as actually  (TRS) enjoys the strong duality \cite{F96,F04,RW}.

When  $\alpha=-\infty$, (GTRS) reduces to the quadratic programming with a single inequality quadratic constraint (QP1QC), see \cite{H,yz} and references therein. Under the primal Slater condition that there is an $\widetilde{x}$ such that $h(\widetilde{x})< \beta$,
the necessary and sufficient optimality
conditions was derived in \cite{M93} and
the strong duality for (QP1QC) is actually due to the well-known S-lemma, see the survey paper \cite{PT}.

When $\alpha=\beta$, (GTRS) is the quadratic programming with a single equality quadratic constraint (QP1EQC). Under the primal Slater condition that there are $x'$ and $x''$ such that $h(x')< \beta<h(x'')$,
the necessary and sufficient optimality condition
was established in \cite{M93}. Suppose $B$ is definite, (QP1EQC) admits the exact semi-definite programming relaxation \cite{yz}. Very recently, the strong duality for (QP1EQC) is guaranteed by the new developed S-lemma with equality
\cite{XWS}.

The two-sided constrained problem (GTRS) was first introduced in \cite{SW}, where $b = 0$ is assumed. Under the further assumption that  $A$ and $B$ are simultaneously diagonalizable via congruence (SDC) \cite{Horn}, the hidden convexity of (GTRS) was observed \cite{B96}.
Very recently, (GTRS) have been extensively and deeply studied \cite{PW}. In particular, strong duality for (GTRS) was established under the following assumptions:
\begin{ass}[\cite{PW}]\label{as0}~
\begin{itemize}
\item[1.] $B \neq 0$.
\item[2.] (GTRS)  is feasible.
\item[3.] The following relative interior constraint qualification holds
\[
{\rm(RICQ)}~  \alpha< B\bullet\widehat{X}+2b^T\widehat{x}+d<\beta,~ {\rm for~ some}~ \widehat{X} \succ \widehat{x}\widehat{x}^T.
\]
\item[4.] (GTRS) is bounded below.
\item[5.] (D-GTRS) is feasible.
\end{itemize}
 \end{ass}
Assumption \ref{as0} is reasonable due to the following facts.
\begin{thm}[\cite{PW}]\label{thm:1}
The following holds for the Items in Assumption \ref{as0}.
\begin{itemize}
\item[(i)] If one of the Items 1, 2, 3 in Assumption \ref{as0} fails, then an explicit solution of (GTRS) can easily be
obtained.
\item[(ii)] If Items 1, 2, 3 in Assumption \ref{as0} hold and  $ b = 0$, then Item 4 implies Item 5.
\item[(iii)] Item 5 in Assumption \ref{as0} implies Item 4.
\end{itemize}
\end{thm}
However, it is still unknown whether Item 4 implies Item 5 when $b\neq 0$, see Remark 2.2 \cite{PW}.

Before presenting the strong duality result, we need some definitions. First, introducing one free Lagrange multiplier $\mu$ yields the
following Lagrange function:
\[
L(x,\mu_+,\mu_-)= f(x) + \mu_{-} (h(x)-\beta)+\mu_{+}(\alpha-h(x)),
\]
where $\mu_{+}=\max\{\mu,0\},\mu_{-}=-\min\{\mu,0\}$.
Then, we can write down the Lagrangian dual problem of (GTRS):
\begin{eqnarray*}
({\rm D\text{-}GTRS})~  \sup_{\mu} \left\{ \inf_x L(x,\mu_+,\mu_-)\right\}
&=&
\sup~  c+\mu d-\mu_{-}\beta+\mu_{+}\alpha-s \\
&&{\rm s.t.}\left[\begin{array}{cc}
A+\mu B& a+\mu b\\
a^T+\mu b^T& s
\end{array}\right]
\succeq0,
\end{eqnarray*}
which is viewed as the dual
semidefinite programming (SDP) relaxation for (GTRS). The primal form of SDP relaxation for (GTRS) can be
obtained by lifting $x\in \Bbb R^n$ to $X:=xx^T\in \Bbb R^{n\times n}$. Relaxing $X=xx^T$ to $X\succeq xx^T$ yields the following primal
SDP relaxation problem:
\begin{eqnarray*}
({\rm SDP\text{-}GTRS})~&\inf&  A\bullet X +2a^{T}x+c \\
&{\rm s.t.}&  \alpha\le  B\bullet X +2b^{T}x+d\le \beta,\\
&& X\succeq xx^T,
\end{eqnarray*}
where the final inequality is equivalent to the linear matrix inequality (LMI)
\[
\left[\begin{array}{cc}1&x^T\\x&X
\end{array}\right]\succeq 0,
\]
according to Schur complement argument.
One can verify that (SDP-GTRS)
is also the conic dual of (D-GTRS).

Let $v(\cdot)$ denote the optimal value of
the problem $(\cdot)$. We have the following strong duality result.
\begin{thm}[\cite{PW}]\label{thm:2}
Under Assumption \ref{as0},  strong duality holds for both ${\rm(GTRS)}$ and ${\rm (SDP\text{-}GTRS)}$, i.e.,
\[
 v{\rm(GTRS)}=v{\rm(D\text{-}GTRS)}=v{\rm(SDP\text{-}GTRS)}.
\]
Moreover, $v{\rm (SDP\text{-}GTRS)}$ is attained.
\end{thm}

In this paper,  Theorems \ref{thm:1} and \ref{thm:2} are both extended. More precisely, we prove that Item 4 implies Item 5 when $b\neq 0$, which answers the open question remained in Theorem \ref{thm:1}. For Theorem \ref{thm:2}, we show that Items 1 and 3 in Assumption \ref{as0} are actually sufficient to guarantee the strong duality for (GTRS). As a by-product, Item 2 is redundant since it can be implied by Item 3. The above new results
are presented in Section 3. Actually, they are applications of the newly developed S-lemma with interval bounds, which is completely
characterized in Section 2. Conclusions are made in Section 4.

Throughout the paper,  the notations $\Bbb R^n$ and $S^n_{+}$ denote the $n$-dimensional vector space and $n\times n$ positive semidefinite symmetric
matrix space, respectively.
Denote by $A\succ(\succeq) 0$ the matrix $A$ is positive  (semi)definite.
The inner product of two matrices $A, B$ is denoted by  $A\bullet B=\sum_{i,j=1}^{n}a_{ij}b_{ij}$. Denote by
$\mathcal{N}(B)$  the null  space of $B$.

\section{S-Lemma and Generalization }

The fundamental S-Lemma was first proved by Yakubovich \cite{y71,y77}
in 1971, see recent surveys \cite{D06,PT}.
\begin{thm}[\cite{y71,y77}]\label{thm:ieq}
Under the Slater assumption that
there is an $\overline{x}\in
\mathbb{R}^n$ such that $h(\overline{x}) < 0$,
the system
\begin{equation}
f(x) < 0,~ h(x)\le 0\nonumber
\end{equation}
is unsolvable
if and only if there is a nonnegative number $\mu\ge0$ such that
\begin{equation}
f(x) + \mu h(x)\ge0, ~\forall x\in \mathbb{R}^n. \nonumber
\end{equation}
\end{thm}
Very recently,
the S-lemma with equality, known as a long-standing open problem, has been proved by Xia et al. \cite{XWS}.
\begin{thm}[\cite{XWS}]\label{thm:eq}
Suppose the Slater assumption for equality holds, that is, there are $x',x''\in \Bbb R^n$ such that $h(x')<0<h(x'')$.
Then, except for the case that
$A$ has exactly one
negative eigenvalue, $B=0$, $b\neq 0$ and
\begin{equation}
\left[\begin{array}{cc}V^TAV&V^T(Ax_0+a)\\
(x_0^TA+a^T)V&f(x_0)\end{array}\right]\succeq 0,\nonumber
\end{equation}
where $x_0=-\frac{d}{2b^Tb}b$, $V\in \Bbb R^{n\times (n-1)}$ is the matrix basis of $\mathcal{N}(b):=\{x:~b^Tx=0\}$,
the system
\[
f(x) < 0,~ h(x)= 0
\]
is unsolvable
if and only if there is a number $\mu$ such that
\[
f(x) + \mu h(x)\ge0, ~\forall x\in \mathbb{R}^n.
\]
\end{thm}
In this section,
as a further extension of Theorems \ref{thm:ieq} and \ref{thm:eq},  we characterize the S-lemma with interval bounds, which asks
when the following two statements are equivalent:
\begin{itemize}
\item[(${\rm S_1}$)] The system
\begin{equation}
f(x) < 0,~ \alpha \leq h(x)\leq \beta\label{eq}
\end{equation}
is unsolvable;
\item[(${\rm S_2}$)] There is a number $\mu\in \mathbb{R}$ such that
\[
f(x) + \mu_{-} (h(x)-\beta)+\mu_{+}(\alpha-h(x))\ge0, ~\forall x\in \mathbb{R}^n.
\]
where $\mu_{+}=\max\{\mu,0\},\mu_{-}=-\min\{\mu,0\}.$
\end{itemize}
Since the special cases $\alpha=-\infty$ (or $\beta=+\infty$) and $\alpha=\beta$ have been settled in Theorems \ref{thm:ieq} and \ref{thm:eq}, respectively,
throughout this paper, we can always make the following assumption:
\begin{ass}\label{as:basic}
$
-\infty<\alpha<\beta<+\infty.
$
\end{ass}
The above S-lemma with interval bounds can be regarded as a special case of the general S-procedure \cite{D06}. Actually, Polyak \cite{P98} succeeded in proving a version of S-procedure involving two quadratic functions in the constraint set:
\begin{thm}[\cite{P98}]\label{thm:pr}
Suppose $n \ge 3$, $f_i (x) = x^TA_ix, i = 0, 1, 2$,  real numbers $\alpha_i$, $i =
0, 1, 2$ and there exist $\mu\in\Bbb R^2, x^0 \in\Bbb R^n$ such that
\begin{eqnarray}
\mu_1A_1 + \mu_2A_2 \succ  0, \nonumber\\
f_1(x^0) < \alpha_1, f_2(x^0) < \alpha_2.\nonumber
\end{eqnarray}
Then the system
\[
f_0(x)<\alpha_0 , ~f_1(x) \le \alpha_1, ~f_2(x)\le\alpha_2
\]
has no solution if and only if there exist $\tau_1 \ge 0, ~\tau_2\ge 0$:
\begin{eqnarray*}
 A_0+ \tau_1A_1 + \tau_2A_2 \succeq 0, \\
\alpha_0 + \tau_1\alpha_1 + \tau_2\alpha_2\le0.
\end{eqnarray*}
\end{thm}
It should be noted that Theorem \ref{thm:pr} only implies a special case of the S-lemma with interval bounds where $a=b=0$, and $B$ is definite.

Now we can establish the general S-lemma with interval bounds.
Without loss of generality, we make the following assumption:
\begin{ass}\label{as1}
There exists an $\overline{x}\in \mathbb{R}^n$ such that $\alpha< h(\overline{x})< \beta$.
\end{ass}
\begin{thm} \label{thm:two}
Under Assumptions \ref{as:basic} and \ref{as1},  S-lemma with interval bounds  holds except that
$A$ has exactly one negative eigenvalue, $B=0$, $b\neq 0$ and there exists a $\nu\geq0$ such that
\begin{equation}
\left[\begin{array}{ccc}
V^TAV                  & \frac{1}{2b^Tb}V^TAb          & V^Ta\\
\frac{1}{2b^Tb}b^TAV   & \frac{b^{T}Ab}{(2b^Tb)^2}+\nu & \frac{a^{T}b}{2b^Tb}-\frac{\nu}{2}(\alpha+\beta-2d)\\
a^TV & \frac{a^{T}b}{2b^Tb}-\frac{\nu}{2}(\alpha+\beta-2d) &
c+\nu(\alpha-d)(\beta-d)
\end{array}\right]\succeq 0, \label{pd}
\end{equation}
where $V\in \Bbb R^{n\times (n-1)}$ is the
matrix basis of $\mathcal{N}(b)$.
\end{thm}
\begin{proof}
%
Note that it is trivial to verify that (${\rm S_2}$) always implies (${\rm S_1}$). It is sufficient to assume (${\rm S_1}$) holds and then show  (${\rm S_2}$) is also true.

We first assume
 \[
 \alpha\leq\inf_{{x\in \mathbb R^{n}}}h(x)\leq\sup_{{x\in \mathbb R^{n}}}h(x)\leq\beta.
 \]
Then, (${\rm S_1}$) becomes that  $f(x)<0$ is unsolvable. It certainly implies (${\rm S_2}$) holds with the setting $\mu=0$.

Next, we assume exactly one of the following case occurs:
\begin{eqnarray*}
  &&\alpha\leq\inf_{{x\in \mathbb R^{n}}}h(x)<\beta<\sup_{{x\in \mathbb R^{n}}}h(x),\\
  &&\inf_{{x\in \mathbb R^{n}}}h(x)<\alpha<\sup_{{x\in \mathbb R^{n}}}h(x)\leq \beta.
\end{eqnarray*}
Without loss of generality, we assume the first case holds.
Consequently,  the system (\ref{eq}) in
(${\rm S_1}$) is equivalent to
\[
f(x) < 0,~ h(x)\leq \beta
\]
and there is an $\widehat{x}\in \mathbb R^{n}$ such that $h(\widehat{x})<\beta$, i.e.,
Slater condition holds.
According to the S-lemma with inequality (i.e., Theorem \ref{thm:ieq}), (${\rm S_1}$) holds if and only if  there is a number $\nu \ge0$ such that
\[
f(x) + \nu (h(x)-\beta)\ge0, ~\forall x\in \mathbb{R}^n.
\]
It follows that (${\rm S_2}$) holds with $\mu=-\nu$, which finishes the proof.
%
%
%

Now, under Assumption \ref{as1}, it is sufficient to assume
\begin{equation}
\inf_{{x\in \mathbb R^{n}}}h(x)<\alpha<\beta<\sup_{{x\in \mathbb R^{n}}}h(x).
\label{ass:le}
\end{equation}
Firstly, we further assume either $A\succeq 0$ or $B\neq 0$.
Suppose (${\rm S_1}$) holds. Then, for any $s\in[\alpha,\beta]$,
the system
\begin{equation*}
f(x) < 0,~  h(x)-s=0,
\end{equation*}
is unsolvable. Assumption (\ref{ass:le}) implies that there are $x',x''\in \Bbb R^n$ such that $h(x')<\alpha<\beta<h(x'')$. It follows that
\[
h(x')-s<\alpha-s\le 0 \le\beta-s<h(x'')-s.
\]
According to Theorem \ref{thm:eq}, there is a number $\mu(s)$ such that
\begin{equation}
f(x) + \mu(s) (h(x)-s)\ge0, ~\forall x\in \mathbb{R}^n.\label{main}
\end{equation}
\begin{itemize}
\item[(a)] Suppose $\mu(\beta)> 0$. Let $\mu=-\mu(\beta)$. Then $\mu_-=\mu(\beta)$ and
\[
f(x)+\mu_-(h(x)-\beta)\ge 0,~\forall x\in \mathbb{R}^n.
\]
\item[(b)]Suppose $\mu(\alpha)< 0$. Let $\mu=-\mu(\alpha)$. Then $\mu_+=-\mu(\alpha)$ and
\[
f(x)+\mu_+(\alpha-h(x))\ge 0,~\forall x\in \mathbb{R}^n.
\]
\item[(c)]Suppose $\mu(\alpha)\geq 0 \geq\mu(\beta)$.  (\ref{main}) implies that
\begin{eqnarray*}
&&f(x) + \mu(\alpha) (h(x)-\alpha)\ge0, ~\forall x\in \mathbb{R}^n,\\
&&f(x) + \mu(\beta) (h(x)-\beta)\ge0, ~\forall x\in \mathbb{R}^n.
\end{eqnarray*}
According to Theorem \ref{thm:ieq}, both the system
\[
f(x)<0,~  h(x)\le \alpha,
\]
and the system
\[
f(x)<0,~  h(x)\ge \beta,
\]
are unsolvable.  Since (${\rm S_1}$) holds, we have
\[
f(x)\ge 0,~\forall x\in \mathbb{R}^n,
\]
(${\rm S_2}$) holds with $\mu=0$.
\end{itemize}
Therefore,
S-lemma with interval bounds  holds under the assumption either $A\succeq 0$ or $B\neq 0$.

Now we assume $A\not \succeq 0$ and $B=0$. Then, (${\rm S_2}$) cannot hold true.
According to  Assumption (\ref{ass:le}), we have $b\neq 0$.
Notice that
\[
\left\{x\in\Bbb R^n:\alpha\leq h(x)\leq \beta\right\}=\left\{\frac{z}{2b^Tb}b+Vy: z\in [\alpha-d,\beta-d],~y\in \Bbb R^{n-1}\right\}
\]
where $V$ is a matrix basis of $\mathcal{N}(b)$.
Trivially, (${\rm S_1}$) holds if and only if
\begin{equation}
\inf_{h(x)\in [\alpha,\beta]} f(x)\ge 0,\nonumber
\end{equation}
or equivalently,
\[
\inf_{\widetilde{h}(z)\le0,~y\in \Bbb R^{n-1}}\left\{ f\left(\frac{z}{2b^Tb}b+Vy\right)
\right\}\ge 0,
\]
where
\[
\widetilde{h}(z):=(z-(\alpha-d))(z-(\beta-d))=z^2-(\alpha+\beta-2d)z+(\alpha-d)(\beta-d).
\]
Therefore, for any given $y\in \Bbb R^{n-1}$,
the system
\[
f\left(\frac{z}{2b^Tb}b+Vy\right)<0,~ \widetilde{h}(z)\leq0
\]
is unsolvable. Since $\alpha<\beta$,  Slater assumption holds for $\widetilde{h}(z)\le0$. According to  Theorem \ref{thm:ieq},  there exists a $\nu\geq0$
such that
\begin{equation}
f\left(\frac{z}{2b^Tb}b+Vy\right)+\nu \widetilde{h}(z)\geq0.\label{pr:ineq}
\end{equation}
Notice that
\[
f\left(\frac{z}{2b^Tb}b+Vy\right)= \frac{b^{T}Ab}{(2b^Tb)^2}z^2+\frac{a^{T}b}{b^Tb}z+
\frac{z}{b^Tb}b^{T}AVy+2a^{T}Vy+y^{T}V^{T}AVy+c.
\]
(\ref{pr:ineq}) can be rewritten as
\[
\left[\begin{array}{c}y\\z\\1\end{array}\right]^T
\left[\begin{array}{ccc}
V^TAV                  & \frac{1}{2b^Tb}V^TAb          & V^Ta\\
\frac{1}{2b^Tb}b^TAV   & \frac{b^{T}Ab}{(2b^Tb)^2}+\nu & \frac{a^{T}b}{2b^Tb}-\frac{\nu}{2}(\alpha+\beta-2d)\\
a^TV & \frac{a^{T}b}{2b^Tb}-\frac{\nu}{2}(\alpha+\beta-2d) &
c+\nu(\alpha-d)(\beta-d)
\end{array}\right]
\left[\begin{array}{c}y\\z\\1\end{array}\right]\ge 0.
\]
Therefore, under the assumption $A\not \succeq 0$ and $B=0$, (${\rm S_1}$) holds if and only if (\ref{pd}) holds. 
Since $A\not\succeq0$ and $V^TAV\succeq 0$,
it must hold that $A$ has exactly one negative eigenvalue.
\end{proof}


\section{ Strong Duality for (GTRS)}

In this section, we apply the S-lemma with interval bounds to establish strong duality for (GTRS).

We first study the relation between Assumptions \ref{as0} and \ref{as1}.
\begin{lem} \label{lem:1}
Assumption \ref{as1} is equivalent to  Item 3 in Assumption \ref{as0}.
\end{lem}
\begin{proof}
Suppose Assumption \ref{as1} is violated, we have either $\inf_{{x\in \mathbb R^{n}}}h(x)\geq\beta$ or $\sup_{{x\in \mathbb R^{n}}}h(x)\leq\alpha$.
We first assume $\inf_{{x\in \mathbb R^{n}}}h(x)\geq\beta$. It follows that   $B\succeq0$. For any $\widehat{X} \succ \widehat{x}\widehat{x}^T$, we have $B\bullet(\widehat{X}-\widehat{x}\widehat{x}^T)\geq0$. If Item 3 in Assumption \ref{as0} holds, we obtain the following contradiction:
\[
h(\widehat{x})= B\bullet\left(\widehat{x}\widehat{x}^T\right)+2b^T\widehat{x}+d\leq B\bullet\widehat{X}+2b^T\widehat{x}+d<\beta.
\]
The other case $\sup_{{x\in \mathbb R^{n}}}h(x)\leq\alpha$ can  be similarly discussed.
Consequently, Item 3 of  Assumption \ref{as0} implies Assumption \ref{as1}.

Now we assume Assumption \ref{as1} holds, i.e., there is an $\widehat{x}$ such that $h(\widehat{x})\in(\alpha,\beta)$. 
Define
\[
\widehat{X}(\epsilon)=\widehat{x}\widehat{x}^T+\epsilon I,
\]
where $I$ is the identity matrix.
Then, we have $\widehat{X}(\epsilon)\succ\widehat{x}\widehat{x}^T$ for all $\epsilon>0$, and
\[
\lim_{\epsilon\rightarrow 0}\left\{ B\bullet\left(\widehat{X}(\epsilon)\right)+2b^T\widehat{x}+d\right\}=
h(\widehat{x})\in (\alpha,\beta).
\]
Therefore, there is an $\epsilon_0>0$ such that $\widehat{X}(\epsilon_0)\succ\widehat{x}\widehat{x}^T$ and
\[
\alpha< B\bullet\widehat{X}(\epsilon_0)+2b^T\widehat{x}+d<\beta.
\]
That is, Items 3 of  Assumption \ref{as0} hold. The proof is complete.
 {\hfill \quad{$\Box$} }

As pointed out by one referee,  Item 2 in Assumption \ref{as0} is unnecessary as it can be implied by Item 3 according to Lemma \ref{lem:1}.

%

Now, as a main result of this paper, we extend Theorem \ref{thm:2}.
\begin{thm} \label{thm:m}
Under Items $1$ and $3$ in Assumption \ref{as0},  strong duality holds for both ${\rm(GTRS)}$ and ${\rm (SDP\text{-}GTRS)}$, i.e.,
\[
 v{\rm(GTRS)}=v{\rm(D\text{-}GTRS)}=v{\rm(SDP\text{-}GTRS)}.
\]
Additionally, suppose Item $4$ in Assumption \ref{as0} holds,
v${\rm(D\text{-}GTRS)}$ is attained.
\end{thm}
{\bf {Proof.}}
According to Lemma \ref{lem:1}, Items $1$ and $3$ in Assumption \ref{as0} imply that $B\neq0$ and Assumption \ref{as1}. It follows from Theorem \ref{thm:two} that S-lemma with interval bounds holds.  Then, we have
\begin{eqnarray}
&&v({\rm GTRS})\nonumber \\
&=&
\sup\limits_{s\in\mathbb{R}}
\left\{s\bigg|
\left\{x\in\mathbb{R}^n|f(x)-s<0,\alpha\leq h(x)\leq \beta\right\}=\emptyset\right\}\nonumber\\
&=& \sup\limits_{s,\mu\in\mathbb{R}}
\left\{s\bigg|
f(x)-s+\mu_{-} (h(x)-\beta)+\mu_{+}(\alpha-h(x))\ge 0,\forall x\in\mathbb{R}^n\right\}\nonumber\\
&=& \sup\limits_{s,\mu\in\mathbb{R}}\left\{s\bigg|
\left[\begin{array}{cc}
A+\mu B& a+\mu b\\
a^T+\mu b^T& c+\mu d-\mu_{-}\beta+\mu_{+}\alpha-s
\end{array}\right]
\succeq0\right\}\label{SDP1}\\
&\le & \inf\limits_{X\in S_+^{n+1}}\left\{
\left[\begin{array}{cc}
A & a \\
a^T & c \end{array}\right]\bullet X \bigg|
 \left[\begin{array}{cc}
 B&  b\\
 b^T& d
\end{array}\right]\bullet X \in [\alpha,\beta], X_{n+1,n+1}=1
\right\}\label{SDP2}\\
&\le & \inf\limits_{x\in\mathbb{R}^n}\left\{
\left[\begin{array}{cc}
A & a \\
a^T & c \end{array}\right]\bullet X \bigg|
 \left[\begin{array}{cc}
 B&  b\\
 b^T& d
\end{array}\right]\bullet X \in [\alpha,\beta], X=\left[\begin{array}{cc}
x \\
 1 \end{array}\right]\left[\begin{array}{cc}
x \\
 1 \end{array}\right]^T
\right\} \nonumber \\
&=&v({\rm GTRS}).\nonumber
\end{eqnarray}
It is not difficult to verify that (\ref{SDP1}) and (\ref{SDP2}) are exactly the dual SDP $({\rm D\text{-}GTRS})$ and  primal SDP $({\rm SDP\text{-}GTRS})$, respectively. Thus, the strong duality holds for both ${\rm(GTRS)}$ and ${\rm (SDP\text{-}GTRS)}$.

Now, suppose Item $4$ in Assumption \ref{as0} also holds, i.e., $v({\rm GTRS})>-\infty$. Then, we have $v{\rm (SDP\text{-}GTRS)}=v({\rm GTRS})>-\infty$.
Note that, according to Item $3$ in Assumption \ref{as0},  $({\rm SDP\text{-}GTRS})$
has a strictly feasible solution.
It follows from the standard strong duality theory for SDP that
$v{\rm(D\text{-}GTRS)}$ is attained.
%
\end{proof}

As an immediate corollary of Theorem \ref{thm:m}, we improve Item (ii) in Theorem \ref{thm:1}, which answers the open question raised in \cite{PW} whether Item 4 implies Item 5 when $b\neq 0$.
\begin{cor}
Under Items $1$ and $3$ in Assumption \ref{as0},  Items 4 and 5 are equivalent.
\end{cor}
\begin{proof}
According to Theorem \ref{thm:m}, under Items $1$ and $3$ in Assumption \ref{as0}, $v{\rm(GTRS)}=-\infty$ if and only if
$v{\rm(D\text{-}GTRS)}=-\infty$, i.e., ${\rm(D\text{-}GTRS)}$ is infeasible.
\end{proof}

Finally,  Theorem \ref{thm:two} implies that  Item 1 in Assumption \ref{as0} is necessary for strong duality. Actually, when $A$ has exactly one negative eigenvalue, $B=0$, $b\neq 0$ and
there is a real number $\nu\geq0$ satisfying (\ref{pd}),
according to the proof of Theorem \ref{thm:two}, we have
\[
 v{\rm(GTRS)}\ge 0,~ v{\rm(D\text{-}GTRS)}=-\infty.
\]
That is, the duality gap is $+\infty$.

However, in the case $B=0$, duality gap can be closed by
reformulating the constraint $\alpha\leq h(x)\leq\beta$ as $(h(x)-\alpha)(h(x)-\beta)\leq0$, which  corresponds to a special case of Theorem \ref{thm:two} where $\alpha=-\infty$.

%

\section{Conclusion}
In this paper, we have extended the classical S-lemma to the interval bounded S-lemma. As an application, we establish strong duality for the interval bounded generalized trust region subproblem (GTRS) under some mild assumptions. Our assumptions are much weaker than that in \cite{PW}. As a by-product, we answer an open question  posted in \cite{PW}. The future work includes further extensions and/or applications of our S-lemma with interval bounds.

\end{document}